\documentclass[letterpaper,12pt]{article}
\usepackage{amsmath,amssymb,amsfonts,epsfig}

\newtheorem{theo}{Theorem}[section]

\newtheorem{lemma}[theo]{Lemma}

\newtheorem{cor}[theo]{Corollary}
\newtheorem{example}[theo]{Example}

\newtheorem{problem}[theo]{Problem}

\newenvironment{proof}{\noindent {\sc Proof}.}
                {\phantom{a} \hfill \framebox[2.2mm]{ } \bigskip}

\newenvironment{proofof}{\noindent {\sc Proof of Theorem}}
                {\phantom{a} \hfill \framebox[2.2mm]{ } \bigskip}

\def\int{{\rm int}}

\newcommand{\eps}{\varepsilon}

\renewcommand{\cal}{\mathcal}

\title{Decomposing complete equipartite multigraphs \\ into cycles of variable lengths: \\ the amalgamation-detachment approach}
\author{Amin Bahmanian\footnote{Mailing address: Department of Mathematics, Illinois State University, Stevenson Hall 313, Campus Box 4520, Normal, Illinois, 61790-4520, USA.} $\,$ and Mateja \v{S}ajna\footnote{Email: msajna@uottawa.ca. Phone: +613-562-5800 ext. 3522. Mailing address: Department of Mathematics and Statistics, University of Ottawa, 585 King Edward Avenue, Ottawa, Ontario, K1N 6N5,Canada.} \\ University of Ottawa}

\begin{document}
\maketitle \baselineskip 17pt

\begin{abstract}
Using the technique of amalgamation-detachment, we show that the complete equipartite multigraph $\lambda K_{n\times m}$ can be decomposed into cycles of lengths $c_1m,\dots,c_km$ (plus a 1-factor if the degree is odd) whenever there exists a decomposition of $\lambda m K_n$ into cycles of lengths $c_1, \dots,c_k$ (plus a 1-factor if the degree is odd). In addition, we give sufficient conditions for the existence of some other, related cycle  decompositions of the complete equipartite multigraph $\lambda K_{n\times m}$.

\end{abstract}

\section{Introduction}

A {\em decomposition} of a graph $\cal G$ is a collection of subgraphs of $\cal G$ whose edge sets partition the edge set of $\cal G$. A graph with a cycle decomposition has no vertices of odd degree, however, a graph in which every vertex has odd degree may admit a decomposition into cycles and a 1-factor. By a {\em $(c_1,c_2,\ldots,c_k)$-cycle decomposition} of a graph $\cal G$ we shall mean a decomposition of $\cal G$ into $k$ cycles of lengths $c_1,\dots,c_k$, respectively, if every vertex in $\cal G$ has even degree, and a decomposition of $\cal G$ into cycles of lengths $c_1,\dots,c_k$ plus a 1-factor if every vertex in $\cal G$ has odd degree.

In this paper, we are concerned with decompositions of complete multigraphs and complete equipartite multigraphs into cycles of variable lengths (plus a 1-factor if the  vertex degrees are odd).  In particular, we show how cycle decompositions of complete multigraphs can be used to obtain cycle decompositions of complete equipartite multigraphs.

Bryant, Horsley, and Pettersson \cite{BryHorsPett} recently proved the following result, which had been conjectured over thirty years ago by Alspach \cite{AlspachConj}.

\begin{theo}\label{BryHorsPett} {\rm \cite{BryHorsPett}}
The complete graph $K_n$ admits a $(c_1,c_2,\ldots,c_k)$-cycle decomposition if and only if $3\leq c_1,\dots,c_k\leq n$ and $\sum_{i=1}^k c_i=  n \lfloor \frac{n-1}{2}\rfloor$.
\end{theo}

Can this result be generalized to complete multigraphs? It is easy to see that if there exists a $(c_1,c_2,\ldots,c_k)$-cycle decomposition of the complete multigraph $\mu K_{n}$, then the following necessary conditions hold:
\begin{description}
\item[(B1)] $2\leq c_i\leq n$ for all $i=1,2,\ldots,k$ and
\item[(B2)] $\sum_{i=1}^k c_i=n \lfloor \frac{\mu (n-1)}{2} \rfloor$.
\end{description}
In the case of cycles of length 2, two more necessary conditions are required (see \cite{BryHorMaeSmi,Bry} and Lemma~\ref{lem:B5}):
\begin{description}
\item[(B3)] If $\mu$ is odd, then $\sum_{c_i \ge 3} c_i \ge n \lfloor \frac{n-1}{2} \rfloor$; and
\item[(B4)] if $\mu$ is even, then $\max\{ c_i: i=1,\ldots,k \} \le \frac{1}{2}\mu\binom{n}{2}-k+2$.
\end{description}
Bryant, Horsley, Maenhaut, and Smith \cite{BryHorMaeSmi} showed that Conditions (B1)--(B3) are also sufficient in the following cases.

\begin{theo}\label{BryHorMaeSmi} {\rm \cite{BryHorMaeSmi}}
Let $\mu$, $n$, $k$, and $c_1 \le  \ldots \le c_k$ be positive integers satisfying Conditions {\rm (B1)--(B3)}, with $n \ge 3$. In addition, assume that
\begin{enumerate}
\item $c_1 \ge \lfloor \frac{n+3}{2} \rfloor$; or
\item $c_k=c_{k-1} \le \lfloor \frac{n+1}{2} \rfloor$; or
\item $c_k=c_{k-1}+1 \le \lfloor \frac{n+2}{2} \rfloor$.
\end{enumerate}
Then the complete multigraph $\mu K_{n}$ admits a $(c_1,c_2,\ldots,c_k)$-cycle decomposition.
\end{theo}
It appears that is has now been proved \cite{Bry} that Conditions (B1)--(B4) are in fact sufficient in all cases.

In this paper, we are concerned with the following generalization of Alspach's Conjecture.

\begin{problem}\label{que:generalize}{\rm
Determine the necessary and sufficient conditions on parameters $\lambda$, $m$, $n$, and $c_1, \dots,c_k$ for the complete equipartite multigraph $\lambda K_{n \times m}$ (with $n$ parts of cardinality $m$) to admit a $(c_1,c_2,\ldots,c_k)$-cycle decomposition.}
\end{problem}
Observe that Conditions (C1)--(C4) below are necessary for the existence of a $(c_1,c_2,\ldots,c_k)$-cycle decomposition of $\lambda K_{n \times m}$. While Conditions (C1)--(C3) are easy to see, Condition (C4) will be proved in Lemma~\ref{lem:B5} in more generality.
\begin{description}
\item[(C1)] $2\leq c_i\leq mn$ for all $i=1,2,\ldots,k$;
\item[(C1$'$)] if $n=2$, then $c_1, \dots,c_k$ are all even;
\item[(C2)] $\sum_{i=1}^k c_i=mn \lfloor \frac{\lambda m (n-1)}{2} \rfloor$;
\item[(C3)] if $\lambda$ is odd, then $\sum_{c_i \ge 3} c_i \ge m^2 \binom{n}{2}$; and
\item[(C4)] if $\lambda$ is even, then $\max\{ c_i: i=1,\ldots,k \} \le \frac{1}{2}\lambda m^2 \binom{n}{2}-k+2$.
\end{description}

Many partial solutions to Problem~\ref{que:generalize} are known for complete equipartite graphs $K_{n \times m}$ of even degree and uniform cycle lengths, that is, for $c_1=\ldots=c_k=c$. Necessary and sufficient conditions have been determined for $c=3$ \cite{Han75}; $c=5$ \cite{BillHoffMae99}; $c \in \{4,6,8 \}$ \cite{BilCav}; prime $c \ge 7$ \cite{ManPau06}; $c$ twice a prime \cite{Smith08JCD}, three times a prime \cite{Smith09AuJC}, and prime square \cite{Smith10JCD}; and for $c$ small relative to the number of parts $m$ \cite{CavSmi10, CavSmi11, Hors12}. Problem~\ref{que:generalize} has also been completely solved for graphs $K_{n \times m}$ of even degree and uniform cycle lengths when the number of parts $n$ is small; namely for $n=2$ \cite{DSot81}; $n=3$ \cite{Cav}, $n=4$ \cite{BilCavSmi09}, and $n=5$ \cite{BilCavSmi10}. For complete equipartite multigraphs $\lambda K_{n \times m}$ of even degree and uniform cycle lengths, Problem~\ref{que:generalize} has been solved for $c=5$ \cite{BillHoffMae99} and $c$ prime \cite{Smi2010}.

Variable (but very specific) cycle lengths in the bipartite graph $K_{m,m}$ were considered in \cite{ArchDebDinGav, ChouFuHu99, ChouFu07, Hors12, MaPuShen}, and in the bipartite multigraph $2K_{m,m}$, in \cite{ChouFuHu00}. The most comprehensive result to date, contained in \cite{Hors12}, solves Problem~\ref{que:generalize} for cycle lengths satisfying $4 \le c_1\le c_2 \le \ldots c_k \le \min(m,3c_{k-1})$.

For complete equipartite multigraphs $\lambda K_{n \times m}$ with $n>2$, the only known result for variable cycle lengths  gives necessary and sufficient conditions when  $\lambda=1$, $\lambda m(n-1)$ is even, and $c_i \in \{ 4,5 \}$ for all $i$ \cite{HuaFu}. No results are known for the case $\lambda >1$, $n>2$, and variable cycle lengths.

\bigskip

The main goal of this paper is to offer a partial solution to Problem~\ref{que:generalize} in the following form.

\begin{theo}\label{mainthmultialspach}
Let $\lambda$, $m$, $n$, and $c_1, \dots,c_k$ be positive integers  such that there exists a $(c_1,c_2,\ldots,c_k)$-cycle decomposition of $\lambda m K_n$. Then the complete equipartite multigraph $\lambda K_{n \times m}$ admits a $(c_1m,c_2m,\ldots,c_km)$-cycle decomposition.
\end{theo}

The following corollary is immediate.

\begin{cor}
Let $\lambda$, $m$, $n$, and $c_1, \dots,c_k$ be positive integers, and let $\mu=\lambda m$. Assume that $\mu$, $n$, $k$, and $c_1, \ldots,c_k$  satisfy the conditions of Theorem~\ref{BryHorMaeSmi}. Then the complete equipartite multigraph $\lambda K_{n \times m}$ admits a $(c_1m,c_2m,\ldots,c_km)$-cycle decomposition.
\end{cor}

This paper is organized as follows. In Section 2 we give the necessary definitions, terminology, and technical tools that will be used in Section 3 to prove our main result, Theorem~\ref{mainthmultialspach}. The techniques used in its proof are taken a step further to construct some other, related cycle decompositions of complete equipartite multigraphs in Section 4. Finally, in Section 5, we give all possible cycle decompositions of the complete multigraphs with at most four vertices, and as a corollary using Theorem~\ref{mainthmultialspach}, all possible decompositions of the complete equipartite multigraphs $\lambda K_{n \times m}$ with at most four parts into cycles of lengths divisible by $m$.

\section{Preliminaries}

All graphs in this paper are assumed to be finite, loopless, and undirected, often with multiple edges. As usual, the symbol $K_n$ denotes the complete graph with $n$ vertices, and $K_{n\times m}$ denotes the complete $n$-partite graph with all parts of cardinality $m$. For any simple graph ${\cal G}$ and positive integer $\lambda$, the symbol $\lambda {\cal G}$ denotes the multigraph with the multiplicity of every edge equal to $\lambda$ and with the underlying simple graph isomorphic to ${\cal G}$.

Let $\mathcal G=(V,E)$ be a graph, and $u$ and $v$  two distinct vertices of $\mathcal G$. Then $d_{\mathcal G}(u)$, $N_{\mathcal G}(u)$, and $m_{\mathcal G}(u,v)$ will denote the degree of $u$, the set of neighbours of $u$, and the number of edges between $u$ and $v$, respectively,  in $\mathcal G$. Similarly, if $U \subseteq V-\{ u \}$, then $m_{\mathcal G}(u,U)$ denotes the number of edges of $\mathcal G$ incident with $u$ and a vertex in $U$. If $V' \subseteq V$, then $\mathcal G -V'$ denotes the graph obtained from $\mathcal G$ by deleting all vertices in $V'$, as well as all edges incident with a vertex in $V'$. If $V'=\{ v \}$, then we write simply $\mathcal G - v$ instead of $\mathcal G - \{ v \}$.

Two concepts will be used in the proof of Theorem \ref{mainthmultialspach} as the main tools. The first is the amalgamation-detachment technique, first developed in \cite{H2, HR}, and more recently surveyed in \cite{BahRod}. Informally speaking, an {\it $\alpha$-detachment} of a graph $\mathcal G$ is any graph obtained by splitting a vertex $\alpha$ of $\mathcal G$ into one or more vertices, and dividing the  edges incident with $\alpha$ among the resulting (sub)vertices. In particular, in an $\alpha$-detachment $\mathcal G'$ of $\mathcal G$ in which we split vertex $\alpha$ into vertices $\alpha$ and $\beta$,  each edge of the form $\{\alpha,u\}$ in $\mathcal G$ will give rise to an edge of the form either $\{\alpha,u\}$ or $\{\beta,u\}$ in $\mathcal G'$.

The following lemma will be crucial in the induction step of the proof of Theorem~\ref{maininductiveseq}.

\begin{lemma} \label{simpleconlemma}
Let  $\mathcal G$ be a connected graph, and $\mathcal G'$ be an $\alpha$-detachment of $\mathcal G$ obtained by splitting a vertex $\alpha$ into two vertices $\alpha$ and $\beta$.
Then $\mathcal G'$ is connected if and only if $1\leq m_{\mathcal G'}(\beta, V({\mathcal H}))<m_{\mathcal G}(\alpha, V({\mathcal H}))$ for some connected component ${\mathcal H}$ of $\mathcal G - \alpha$.
\end{lemma}

\begin{proof}
Let  $\mathcal G$ and $\mathcal G'$ be as in the statement of the lemma. Note that the connected components of $\mathcal G'-\{ \alpha, \beta \}$ are precisely the connected components of $\mathcal G - \alpha$. If $\mathcal G'$ is connected, then some component ${\mathcal H}$ of $\mathcal G'-\{ \alpha, \beta \}$ must contain neighbours of both $\alpha$ and $\beta$, and conversely. In other words, $\mathcal G'$ is connected if and only if $m_{\mathcal G'}(\alpha,V({\mathcal H})) \ge 1$ and $m_{\mathcal G'}(\beta,V({\mathcal H}))\ge 1$ for some connected component ${\mathcal H}$ of $\mathcal G - \alpha$. Since $m_{\mathcal G}(\alpha,V({\mathcal H}))=m_{\mathcal G'}(\alpha,V({\mathcal H}))+ m_{\mathcal G'}(\beta,V({\mathcal H}))$, the statement of the lemma then follows immediately.
\end{proof}

The second main tool in the proof of Theorem \ref{mainthmultialspach} is the concept of edge colouring, in particular, de Werra's Theorem~\ref{BEElemma} below.
A {\em $k$-edge-colouring} of a graph $\mathcal G=(V,E)$ is a mapping $f: E\rightarrow K$, where $K=\{1,\ldots,k\}$ is the set of $k$ \textit{colours}. For any $i \in K$, the symbol $\mathcal G(i)$ will denote the spanning subgraph of $\mathcal G$ whose edge set is the set of all edges of colour $i$; we call such a spanning subgraph a \textit{colour class} of $\mathcal G$ with respect to the edge colouring $f$. Observe that a colour class may have (many)  isolated vertices.

A $k$-edge-colouring of a graph $\mathcal G=(V,E)$ is called {\it equitable} if $|d_{\mathcal G(i)}(u)-d_{\mathcal G(j)}(u)|\leq 1$ for all $i,j \in K$ and $u\in V$; that is, if every vertex is incident with ``almost the same'' number of edges of each colour.

The following extremely useful result by de Werra \cite{deWerra75BEE} guarantees existence of an equitable $k$-edge-colouring in any bipartite graph. For completeness, and since publication \cite{deWerra75BEE} is not available to us, we present a  proof.

\begin{theo}\label{BEElemma} {\rm  \cite{deWerra75BEE}}
Let $\mathcal G=(V,E)$ be a bipartite graph and $k$ a positive integer. Then $\mathcal G$ admits an equitable $k$-edge-colouring.
\end{theo}

\begin{proof}
The assertion clearly holds for $k=1$, hence we may assume $k \ge 2$.

For a  graph $\cal G=(V,E)$ and a fixed $k$-edge-colouring $f: E \rightarrow K$ of $\cal G$, let $d_i(v)=d_{\cal G(i)}(v)$ and $d_{ij}(v)=|d_i(v)-d_j(v)|$ for all $v \in V$ and $i,j \in K$. Define a parameter $\delta(\cal G,f)$ as follows:
$$\delta(\cal G,f)=\sum_{v \in V} \sum_{i,j \in K} \left(  d_{ij}(v) + |d_{ij}(v)-1| -1 \right).$$

{\it Claim:} $\delta(\cal G,f)=0$ if and only if $f$ is an equitable $k$-edge-colouring of $\cal G$.

{\it Proof of the claim:} Assume $f$ is an equitable $k$-edge-colouring of $\cal G$. Then, for all $v \in V$ and $i,j \in K$, we have $d_{ij}(v) \in \{ 0,1 \}$, and it easily follows that $\delta(\cal G,f)=0$.

Conversely, let $f$ be a $k$-edge-colouring of $\cal G$ with $\delta(\cal G,f)=0$. First observe that, for any  real number $s$, the quantity $|s| + |s-1|$ gives the sum of distances of $s$ from 0 and 1, and hence $|s| + |s-1| -1 \ge 0$. Thus $\sum_{v \in V} \sum_{i,j \in K} \left(  d_{ij}(v) + |d_{ij}(v)-1| -1 \right)=0$ implies that $d_{ij}(v) + |d_{ij}(v)-1|=1$ for all $v \in V$ and $i,j \in K$. Since $d_{ij}(v)$ is a non-negative integer, it follows that $d_{ij}(v) \in \{ 0,1 \}$, in other words, $f$ is an equitable $k$-edge-colouring of $\cal G$.

\medskip

Now suppose $\cal G$ is a bipartite graph that admits no equitable $k$-edge-colouring. Let $f$ be a $k$-edge-colouring of $\cal G$ that minimizes $\delta(\cal G,f)$. Since $f$ is not equitable, there exist a vertex $u \in V$ and colours $s,t \in K$ such that $d_s(u)-d_t(u) \ge 2$. Let $T$ be a maximal trail in ${\cal G}$ with initial vertex $u$, first edge in $\cal G(s)$, and edges alternately in $\cal G(s)$ and $\cal G(t)$. (Recall that a {\em trail} is an alternating sequence $v_0e_1v_1\ldots v_{m-1}e_mv_m$ of vertices and edges such that each edge $e_i$ has endpoints $v_{i-1}$ and $v_i$, and no edge in the sequence is repeated.) For any internal vertex $v$ of $T$, the trail enters $v$ with an edge of colour $s$ and exits with an edge of colour $t$, or vice-versa. Since $\cal G$ is bipartite, $d_s(u) \ne d_t(u)$, and $T$ is maximal, the trail $T$ cannot be closed. Thus its terminal vertex, call it  $z$, is distinct from $u$.

We construct a new $k$-edge-colouring $f'$ of ${\cal G}$ by swapping colours $s$ and $t$ of $f$ along the trail $T$. With respect to this new colouring $f'$, let $d_i'(v)=d_{\cal G(i)}(v)$ and $d_{ij}'(v)=|d_i'(v)-d_j'(v)|$ for all $v \in V$ and $i,j \in K$.

We now show that
\begin{eqnarray}\label{ineq1}
\sum_{i,j \in K} \left(  d_{ij}'(v) + |d_{ij}'(v)-1| -1 \right) \le \sum_{i,j \in K} \left(  d_{ij}(v) + |d_{ij}(v)-1| -1 \right)
\end{eqnarray}
for all $v \in V$, with strict inequality when $v=u$. This is obvious (with equality) for vertices $v \not\in \{ u,z \}$ since $d_i'(v)=d_i(v)$ for all $i \in K$. It is also clear that for $v \in \{ u,z \}$,
\begin{eqnarray}\label{ineq2}
d_{ij}'(v) + |d_{ij}'(v)-1| -1 \le d_{ij}(v) + |d_{ij}(v)-1| -1
\end{eqnarray}
 holds (with equality) for all $i,j \in K - \{ s,t \}$.

We now verify that strict inequality holds in (\ref{ineq2}) for $v=u$ and $\{ i,j \}=\{ s, t\}$. Since $d_s(u)-d_t(u) \ge 2$, we have $d_{st}'(u)=d_{st}(u)-2$. Note that as a consequence, $|d_{st}'(u)-1|=|d_{st}(u)-1|$ if $d_{st}(u)=2$, and $|d_{st}'(u)-1|=|d_{st}(u)-1|-2$ otherwise. In any case,
$$d_{st}'(u) + |d_{st}'(u)-1|-1 < d_{st}(u) + |d_{st}(u)-1|-1.$$

Next, we verify Inequality~(\ref{ineq2}) for $v=z$ and $\{ i,j \}=\{ s, t\}$. Since $T$ is a maximal trail alternating colours $s$ and $t$, for the terminal vertex $z$ of $T$, we must have either $d_s(z)>d_t(z)$ or $d_s(z)<d_t(z)$. Assume $d_s(z)>d_t(z)$. Then the last edge of the trail $T$ must be of colour $s$, and will be swapped to colour $t$ in $f'$. Hence $d_{st}'(z)=d_{st}(z)$ if $d_{st}(z)=1$, and $d_{st}'(z)=d_{st}(z)-2$ otherwise. Furthermore, $|d_{st}'(z)-1|=|d_{st}(z)-1|$ if $d_{st}(z) \in \{ 1,2 \}$, and $|d_{st}'(z)-1|=|d_{st}(z)-1|-2$ otherwise. In any case,
$$d_{st}'(z) + |d_{st}'(z)-1|-1 \le d_{st}(z) + |d_{st}(z)-1|-1.$$
A similar argument shows that Inequality~(\ref{ineq2}) holds for $v=z$ and $\{ i,j \}=\{ s, t\}$ when $d_s(z)<d_t(z)$.

Finally, we'll show that
\begin{eqnarray}\label{ineq3}
d_{s\ell}'(v) + |d_{s\ell}'(v)-1| + d_{t\ell}'(v) + |d_{t\ell}'(v)-1| \le d_{s\ell}(v) + |d_{s\ell}(v)-1| + d_{t\ell}(v) + |d_{t\ell}(v)-1|
\end{eqnarray}
for $v \in \{ u,z \}$ and any $\ell \in K-\{ s,t \}$.

Take any colour $\ell \in K-\{ s,t \}$. Since  $d_s'(u)=d_s(u)-1$ and $d_t'(u)=d_t(u)+1$, we have
$$d_{s\ell}'(u)=\left\{ \begin{array}{ll}
                        d_{s\ell}(u)-1 & \mbox{ if } d_\ell(u) < d_s(u) \\
                        d_{s\ell}(u)+1 & \mbox{ if } d_\ell(u) \ge d_s(u)
                        \end{array} \right.
\quad \mbox{ and } \quad
d_{t\ell}'(u)=\left\{ \begin{array}{ll}
                        d_{t\ell}(u)-1 & \mbox{ if } d_\ell(u) > d_t(u) \\
                        d_{t\ell}(u)+1 & \mbox{ if } d_\ell(u) \le d_t(u)
                        \end{array} \right..$$
Therefore, since $d_t(u)< d_s(u)$, we can see that
$$d_{s\ell}'(u) + d_{t\ell}'(u)  \le d_{s\ell}(u) + d_{t\ell}(u).$$               Furthermore, since $d_{s\ell}'(u)=d_{s\ell}(u)\pm 1$ and $d_{t\ell}'(u)=d_{t\ell}(u) \pm 1$, we also have
$|d_{s\ell}'(u)-1|=|d_{s\ell}(u)-1| \pm 1$ and
$|d_{t\ell}'(u)-1|=|d_{t\ell}(u)-1| \pm 1$.  Thus
$$|d_{s\ell}'(u)-1| + |d_{t\ell}'(u)-1| \le |d_{s\ell}(u)-1| + |d_{t\ell}(u)-1|$$
unless both $|d_{s\ell}'(u)-1|=|d_{s\ell}(u)-1| + 1$ and
$|d_{t\ell}'(u)-1|=|d_{t\ell}(u)-1| + 1$.
Now $|d_{s\ell}'(u)-1|=|d_{s\ell}(u)-1| + 1$ if and only if either $d_{s\ell}(u)=1$ or both $d_{s\ell}(u) \ge 2$ and $d_{s\ell}'(u)=d_{s\ell}(u)+1$; that is, if and only if either $d_{s\ell}(u)=1$ or both $d_{s\ell}(u) \ge 2$ and $d_{\ell}(u) \ge d_{s}(u)$. Similarly, $|d_{t\ell}'(u)-1|=|d_{t\ell}(u)-1| + 1$ if and only if either $d_{t\ell}(u)=1$ or both $d_{t\ell}(u) \ge 2$ and $d_{\ell}(u) \le d_{t}(u)$. The only possibility then is $d_{s\ell}(u)=d_{t\ell}(u)=1$. In this case, we must have $d_s(u)=d_{\ell}(u)+1$, $d_t(u)=d_{\ell}(u)-1$, and $d_s'(u)=d_{\ell}'(u)=d_t'(u)$, and hence
\begin{eqnarray*}
&& d_{s\ell}'(u) + |d_{s\ell}'(u)-1| + d_{t\ell}'(u) + |d_{t\ell}'(u)-1|   \\
& = & 0 + (|d_{s\ell}(u)-1|+1)+ 0 + (|d_{t\ell}(u)-1|+1) \\
& = & d_{s\ell}(u) + |d_{s\ell}(u)-1| + d_{t\ell}(u) + |d_{t\ell}(u)-1|.
\end{eqnarray*}
Hence Inequality~(\ref{ineq3}) holds for $v=u$. Similarly, reversing the roles of colours $s$ and $t$ if necessary, we can show that it holds for $v=z$.

We have thus shown that (\ref{ineq1}) holds for all $v \in V$, with strict inequality when $v=u$. We conclude that $f'$ is a $k$-edge-colouring of $\cal G$ with $\delta(\cal G,f')< \delta(\cal G,f)$, a contradiction. Hence $\cal G$ must possess an equitable $k$-edge-colouring.
\end{proof}

\section{Proof of the main result}

Throughout the rest of this paper, unless otherwise specified, $\lambda$, $\mu$, $m$, $n$, $k$, and $c_1,c_2, \ldots, c_k$ will denote positive integers,

We first give a simple lemma that implies the necessary condition (C4) for existence of a $(c_1,\dots,c_k)$-cycle decomposition of $\lambda K_{n \times m}$.

\begin{lemma}\label{lem:B5} {\rm \cite{Bry}}
Let ${\cal G}$ be a multigraph in which each edge has even multiplicity, and assume that ${\cal G}$ admits a $(c_1,\dots,c_k)$-cycle decomposition. Then $$\max\{ c_i: i=1,\ldots,k \} \le \frac{1}{2}|E({\cal G})|-k+2.$$
\end{lemma}

\begin{proof}
Suppose that the result does not hold, and let ${\cal G}$ be a smallest counterexample. That is, ${\cal G}$ is a multigraph with the smallest number of edges such that every edge of ${\cal G}$ has even multiplicity, and ${\cal G}$ admits a $(c_1,\dots,c_k)$-cycle decomposition ${\cal C}$ with $\max\{ c_i: i=1,\ldots,k \} > \frac{1}{2}\eps -k+2$. Let $E=E(\cal G)$ and $\eps=|E|$, and let $C_k$ be a cycle in ${\cal C}$ of maximum length $c_k$. For each $e \in E(C_k)$, choose an edge $e'$ parallel to $e$, $e' \ne e$, and let $E^*=E(C_k) \cup \{ e': e \in E(C_k)\}$. Since $c_k > \frac{1}{2}\eps -k+2$, we have $|E-E^*|=\eps-2c_k<2(k-2)<2(k-1)$. Hence there exists a cycle $C$ in  ${\cal C}$, $C \ne C_k$, that contains at most one edge of $E-E^*$. If $C$ contains no edges of $E-E^*$, then $C$ is a $c_k$-cycle parallel to $C_k$, $E(C)=E^*-E(C_k)$, and each of the remaining $k-2$ cycles of ${\cal C}$ contains at least 2 edges of $E-E^*$. It follows that $\eps \ge 2c_k + 2(k-2)$, contradicting $c_k > \frac{1}{2}\eps -k+2$.

Hence $C$ must contain exactly one edge of $E-E^*$ --- call it $e$ --- and the edges in $E(C) \cap E^*$ form a path $P$ of length $t$. Let $P_k$ be the path in $C_k$ parallel to $P$. Obtain a graph ${\cal G'}$ from  ${\cal G}$ by deleting the edges of $P$ and $P_k$, and a cycle decomposition ${\cal C'}$ of ${\cal G'}$ by deleting $C$ from ${\cal C}$ and replacing $C_k$ with the cycle $(C_k-P_k)+e$.  Observe that ${\cal G'}$ has $\eps'=\eps-2t$ edges, and each edge has even multiplicity. Moreover, ${\cal C'}$ is indeed a cycle decomposition of ${\cal G'}$; it contains $k'=k-1$ cycles, and maximum cycle length is $c_k' \ge c_k-t+1$. Hence, by assumption,
$$c_k' \ge c_k-t+1 > \left(\frac{1}{2}\eps -k+2 \right)-t+1 = \frac{1}{2}\eps' -k'+2.$$
Thus ${\cal G'}$ is a smaller counterexample, contradicting the minimality of ${\cal G}$.

We conclude that the statement of the lemma holds.
\end{proof}

Our main Theorem \ref{mainthmultialspach} will follow easily from the seemingly stronger Theorem~\ref{maininductiveseq} and Corollary~\ref{cor:maininductiveseq} below.

\begin{theo}\label{maininductiveseq}
Let $\lambda m(n-1)$ be even, and assume there exists a $(c_1,\dots,c_k)$-cycle decomposition of $\lambda m K_n$.
Then, for all $\ell=n,n+1, \ldots,mn$ there exist a graph $\mathcal G=(V,E)$ of order $\ell$ and a function $g:V \rightarrow {\mathbb Z}^+$ with the following properties:
\begin{description}
\item[(P1)] $\mathcal G$ is $n$-partite;
\item[(P2)] $\sum_{v\in W}g(v)=m$ for each part $W$ of $\mathcal G$;
\item[(P3)] $m_{\mathcal G}(u,v)=\lambda g(u)g(v)$ for each pair of vertices $u,v$ from distinct parts of $\mathcal G$;
\item[(P4)] $\mathcal G$ admits a $k$-edge-colouring such that, for each each colour $i \in \{ 1,2,\ldots,k \}$:
\begin{description}
\item[(P4a)] colour class $\cal G(i)$ has $c_im$ edges;
\item[(P4b)] $d_{{\mathcal G}(i)}(v)\in\{0,2g(v)\}$ for each $v\in V$; and
\item[(P4c)] $\cal G(i)$ has a unique non-trivial connected component.
\end{description}
\end{description}
\end{theo}

\begin{proof} We prove the theorem by induction on $\ell$.

First we prove the basis of induction, case $\ell=n$. Let $\mathcal G=(V,E)=\lambda m^2K_n$ and $g(v)=m$ for all $v\in V$. Then the graph $\mathcal G$ is of order $\ell$, and Properties (P1)--(P3) clearly hold for $\mathcal G$ and the function $g$. By assumption, there exists a decomposition of $\lambda m K_n$ into cycles of lengths $c_1,\ldots, c_k$. Replacing each edge in this decomposition by $m$ parallel edges we obtain a decomposition of $\mathcal G$ into $m$-fold cycles of lengths $c_1,\ldots, c_k$. Now define a $k$-edge-colouring of $\mathcal G$ by taking the colour class $\cal G(i)$ to be the $m$-fold cycle of length $c_i$ in this decomposition, together with the remaining $n-c_i$ isolated vertices. Clearly, Property (P4) then holds for $\mathcal G$ and $g$ as well.

Suppose now that for some $\ell \in \{ n, n+1,  \ldots, mn-1 \}$ there exist a graph $\mathcal G=(V,E)$ of order $\ell$ and a function $g:V \rightarrow {\mathbb Z}^+$ satisfying properties (P1)--(P4) from the statement of the theorem. We shall now construct a graph $\mathcal G'$ of order $\ell+1$ and a function $g':V(\mathcal G') \rightarrow {\mathbb Z}^+$ satisfying Properties (P1)--(P4). Since $\ell<mn$ and (P1)--(P2) hold for $\cal G$, there exists a vertex $\alpha$ of $\mathcal G$ with $g(\alpha)>1$. The graph $\mathcal G'$ will be constructed as an $\alpha$-detachment of $\mathcal G$ with the help of an auxiliary bipartite graph $B$ defined as follows.

First, define sets $K=\{ 1,2,\ldots,k \}$, $X=\{x_1,\dots,x_k\}$, and $V_{\alpha}=V - \{ \alpha \}$, and let $B$ be the bipartite graph with bipartition $\{X, V_{\alpha} \}$ and with $m_B(x_i,u)=m_{\mathcal G(i)}(\alpha,u)$ for each $u\in V_{\alpha}$ and $i \in K$. Observe that, by the induction hypothesis,  $d_{B}(x_i)=d_{\mathcal G(i)}(\alpha)\in\{0,2g(\alpha)\}$  and $d_B(u)=m_{\mathcal G}(\alpha,u)=\lambda g(\alpha)g(u)$ for all $i \in K$ and $u\in V_{\alpha}$ such that $u$ and $\alpha$ are from distinct parts of $\mathcal G$.

By Theorem~\ref{BEElemma}, there exists an equitable $g(\alpha)$-edge-colouring of $B$. With respect to such a colouring  we have  $d_{B(j)}(x_i)=d_B(x_i)/g(\alpha)=d_{\mathcal G(i)}(\alpha)/g(\alpha)\in\{0,2\}$  and $d_{B(j)}(u)=d_B(u)/g(\alpha)=\lambda g(u)$ for all $i \in K$, $j \in \{ 1,2,\ldots,g(\alpha) \}$, and $u\in V_{\alpha}$, where $u$ and $\alpha$ are from distinct parts of $\mathcal G$. In particular, observe that $d_{B(j)}(x_i)$ is constant  with respect to parameter $j$ (namely, it is 0 if $d_B(x_i)=0$, and 2 if $d_B(x_i)=2g(\alpha)$). We shall use one colour class of this equitable $g(\alpha)$-edge-colouring of $B$ to define the $\alpha$-detachment ${\cal G'}$ of ${\cal G}$, however, to guarantee Property (P4c), we may need to first modify the colouring as follows.

Let $B_2$ be a spanning subgraph of $B$ that is the union of two arbitrary colour classes of $B$ with respect to our equitable $g(\alpha)$-edge-colouring.  Then $d_{B_2}(x_i)\in\{0,4\}$ and $d_{B_2}(u)=2\lambda g(u)$ for all $i \in K$ and $u\in V_{\alpha}$ such that $u$ and $\alpha$ are from distinct parts of $\mathcal G$.

Let $K'$ be the subset of $K$ containing all colours $i$ such that
\begin{equation} \label{conntrick}
\mbox{ there exists a connected component } {\mathcal H_i} \mbox{ of } \mathcal G(i)-\alpha \mbox{ with }  m_{B_2}(x_i,V({\mathcal H_i}))=2.
\end{equation}
We form a new (bipartite) graph $B_2'$ from $B_2$ by splitting each vertex $x_i$, for $i \in K'$, into vertices $x_i$ and $y_i$, and then divide the edges incident with $x_i$  so that $m_{B_2'}(x_i, V({\mathcal H_i}))=2$. Theorem~\ref{BEElemma} gives existence of an equitable $2$-edge-colouring of $B_2'$. Take an arbitrary colour class in this colouring of $B_2'$, and obtain a new graph $B_1$ from this colour class by identifying vertices $x_i$ and $y_i$ for each $i \in K'$; call the new vertex $x_i$. Observe that $m_{B_1}(x_i,V({\mathcal H_i}))=1$ for all $i \in K'$, while  $d_{B_1}(x_i)\in\{0,2\}$ for all $i \in K$ and $d_{B_1}(u)=\lambda g(u)$ for all $u\in V_{\alpha}$ such that $u$ and $\alpha$ are from distinct parts of $\mathcal G$.

We are now ready to define the new graph $\mathcal G'$. Informally speaking,  $\mathcal G'$ is obtained from $\mathcal G$  by splitting the vertex $\alpha$ into vertices $\alpha$ and $\beta$, and converting all edges of the form $\{ \alpha, u \}$ that correspond to edges of $B_1$ to edges of the  form $\{ \beta, u \}$, preserving the colour of each edge. More formally, take any $\beta \not\in V$, and define $\mathcal G'$ as a $k$-edge-coloured graph  with $V(\mathcal G')=V\cup\{\beta\}$ and, for all $i \in K$ and $u,v \in V_{\alpha}$,
\begin{eqnarray*}
m_{{\cal G}'(i)}(u,v) &=& m_{{\cal G}(i)}(u,v), \\
m_{{\cal G}'(i)}(\beta,u) &=& m_{B_1}(x_i,u), \\
m_{{\cal G}'(i)}(\alpha,u) &=& m_{{\cal G}(i)}(\alpha,u) - m_{B_1}(x_i,u),  \mbox{ and} \\
m_{{\cal G}'(i)}(\alpha,\beta) &=& 0.
\end{eqnarray*}
Clearly, $\mathcal G'$ is of order $\ell+1$ and is $n$-partite (with $\alpha$ and $\beta$ in the same part). Moreover, $|E(\mathcal G'(i))|=|E(\mathcal G(i))|=c_im$ for all $i \in K$, so Properties (P1) and (P4a) hold for $\mathcal G'$.

We define the function $g':V(\mathcal G') \rightarrow {\mathbb Z}^+$ as follows: $g'(\alpha)=g(\alpha)-1, g'(\beta)=1$, and  $g'(v)=g(v)$ for all $v \in V_{\alpha}$. We then immediately obtain $\sum_{v\in W}g'(v)=m$ for each part $W$ of $\mathcal G'$, so Property (P2) holds for $\mathcal G'$ and $g'$ as well.

To verify Property (P4b), take any $i \in K$. Observe that $d_{{\mathcal G}'(i)}(v)=d_{{\mathcal G}(i)}(v) \in \{0,2g'(v)\}$ for each $v\in V_{\alpha}$. Furthermore, since $d_{B_1}(x_i)\in \{ 0,2 \}$, we have
$d_{{\cal G}'(i)}(\beta)\in\{0,2\}=\{0,2g'(\beta)\}$  and
$d_{{\cal G}'(i)}(\alpha)\in\{0,2g(\alpha)-2\}=\{0,2g'(\alpha)\}$.

To verify Property (P3), first observe that for any $u,v \in V_{\alpha}$ that belong to distinct parts of ${\cal G}'$, and hence to distinct parts of ${\cal G}$, we have
$m_{{\cal G}'}(u,v)=m_{{\cal G}}(u,v)=\lambda g(u)g(v)=\lambda g'(u)g'(v)$. Furthermore, for any $u \in V_{\alpha}$ not in the same part as $\alpha$ and $\beta$, we have
$m_{{\cal G}'}(\beta,u) = d_{B_1}(u)=\lambda g(u)=\lambda g'(u)g'(\beta)$, and
$m_{{\cal G}'}(\alpha,u) = m_{{\cal G}}(\alpha,u)-d_{B_1}(u)=\lambda g(\alpha)g(u)-\lambda g(u)=\lambda (g(\alpha)-1)g(u)=\lambda g'(\alpha)g'(u)$.

It remains to verify Property (P4c), namely, that every colour class ${\mathcal G}'(i)$ has a unique non-trivial connected component. Fix a colour $i \in K$. If $\alpha$ is an isolated vertex in ${\mathcal G}(i)$, then ${\mathcal G}'(i)$ was obtained from ${\mathcal G}(i)$ by adjoining a new isolated vertex $\beta$; hence ${\mathcal G}'(i)$ has a unique non-trivial connected component since ${\mathcal G}(i)$ does.

Hence assume $\alpha$ is a vertex in ${\cal C}(i)$, the unique non-trivial connected component of $\mathcal G(i)$. Let ${\cal C}'(i)$ be the subgraph of ${\mathcal G}'(i)$ induced by $V({\cal C}(i)) \cup \{ \beta \}$. It suffices to show that ${\cal C}'(i)$ is connected; since ${\mathcal G}'(i)$ inherited all isolated vertices of ${\mathcal G}(i)$, it will then follow that ${\cal C}'(i)$ is the unique non-trivial connected component of ${\mathcal G}'(i)$.

First observe that, since $\mathcal C(i)$ is a connected even graph, it has no cut edges, which implies that for each connected component $\mathcal H_i$ of $\mathcal C(i) - \alpha$ we have $m_{{\mathcal C}(i)}(\alpha,V(\mathcal H_i))\ge 2$, and hence also $m_{B}(x_i,V(\cal H_i)) \ge 2$. This fact will be used in Cases 1--3 below.

By Lemma~\ref{simpleconlemma}, it suffices to show that for some connected component $\mathcal H_i$ of $\mathcal C(i)-\alpha$ we have $1\le m_{{\cal C}'(i)}(\beta, V(\cal H_i)) < m_{{\cal C}(i)}(\alpha, V(\cal H_i))$,
or equivalently, that $1\le m_{B_1}(x_i,V(\cal H_i)) < m_B(x_i,V(\cal H_i))$.
Since $d_{B_2}(x_i)=4$, there are three cases to consider.

{\it Case 1: $m_{B_2}(x_i,V(\cal H_i))\in\{3,4\}$ for some connected component  $\cal H_i$ of $\mathcal C(i)-\alpha$.} Then
$$1\leq m_{B_1}(x_i,V(\cal H_i))\leq 2<m_{B_2}(x_i,V(\cal H_i))\leq m_{B}(x_i,V(\cal H_i)).$$

{\it Case 2: $m_{B_2}(x_i,V(\cal H_i))=2$ for some connected component  $\cal H_i$ of $\mathcal C(i)-\alpha$.} Since connected components of $\mathcal C(i)-\alpha$ are precisely the connected components of $\mathcal G(i)-\alpha$, we have $i\in K'$. We may assume that the graph $B_1$ was constructed using this particular connected component $\cal H_i$, so that $m_{B_1}(x_i,V(\cal H_i))=1$, which implies that $1=m_{B_1}(x_i,V(\cal H_i))<2\leq m_{B}(x_i,V(\cal H_i))$.

{\it Case 3: $m_{B_2}(x_i,V(\cal H_i)) = 1$ for four connected components  $\cal H_i$ of $\mathcal C(i)-\alpha$.} Now two of these four components have the property that $m_{B_1}(x_i,V(\cal H_i))=1 < 2 \le m_{B}(x_i,V(\cal H_i))$.

Since the conditions of Lemma~\ref{simpleconlemma} are satisfied in all cases, we conclude that ${\cal C}'(i)$ is connected, proving Property (P4c) for ${\mathcal G}'$.

We have thus shown that ${\mathcal G}'$ and $g'$ satisfy Properties (P1)--(P4). The result follows by induction.
\end{proof}

In the next corollary, we shall extend Theorem~\ref{maininductiveseq} to multigraphs of odd degree.

\begin{cor}\label{cor:maininductiveseq}
Let $\lambda m(n-1)$ be odd, and assume there exists a $(c_1,\dots,c_k)$-cycle decomposition of $\lambda m K_n$. Then, for all $\ell=n,n+1, \ldots,mn$ there exist a graph $\mathcal G=(V,E)$ of order $\ell$ and a function $g:V \rightarrow {\mathbb Z}^+$ with the following properties:
\begin{description}
\item[(P1)] $\mathcal G$ is $n$-partite;
\item[(P2)] $\sum_{v\in W}g(v)=m$ for each part $W$ of $\mathcal G$;
\item[(P3)] $m_{\mathcal G}(u,v)=\lambda g(u)g(v)$ for each pair of vertices $u,v$ from distinct parts of $\mathcal G$;
\item[(P4)] $\mathcal G$ admits a $(k+1)$-edge-colouring such that for each each colour $i \in \{ 1, \ldots, k \}$:
\begin{description}
\item[(P4a)] colour class $\cal G(i)$ has $c_im$ edges;
\item[(P4b)] $d_{{\mathcal G}(i)}(v)\in\{0,2g(v)\}$ for each $v\in V$; and
\item[(P4c)] $\cal G(i)$ has a unique non-trivial connected component.

\hspace{-10mm} In addition,
\item[(P4d)] colour class $\cal G(k+1)$ has $\frac{1}{2}mn$ edges; and
\item[(P4e)] $d_{{\mathcal G}(k+1)}(v)=g(v)$ for each $v\in V$.
\end{description}
\end{description}
\end{cor}

\begin{proof}
The proof of this corollary is very similar to the proof of Theorem~\ref{maininductiveseq}, hence we highlight only the differences.

In the base case $\ell=n$, the colour class ${\cal G}(k+1)$ is defined as the $m$-fold 1-factor in ${\cal G}=\lambda m^2$ arising from the 1-factor in the presumed decomposition of  $\lambda m K_n$. Properties (P1)--(P4) then clearly hold.

In the induction step, we define the auxiliary bipartite graph $B$ as follows. First, let $K^*=K \cup \{ k+1 \}$ and $X^*=X \cup \{ x_{k+1} \}$, and let $B$ be the bipartite graph with bipartition $\{X^*, V_{\alpha} \}$ and with $m_B(x_i,u)=m_{\mathcal G(i)}(\alpha,u)$ for each $u\in V_{\alpha}$ and $i \in K^*$. By the induction hypothesis,  $d_{B}(x_i)=d_{\mathcal G(i)}(\alpha)\in\{0,2g(\alpha)\}$ for all $i \in K$, $d_{B}(x_{k+1})=d_{\mathcal G(k+1)}(\alpha)=g(\alpha)$, and $d_B(u)=m_{\mathcal G}(\alpha,u)=\lambda g(\alpha)g(u)$ for all $u\in V_{\alpha}$.

As in the proof of Theorem~\ref{maininductiveseq}, we take an equitable $g(\alpha)$-edge-colouring of $B$, and let $B_2$ be the union of two of its colour classes. The subgraph $B_1$ is then defined exactly as before, adjusting only the edges incident with the vertices in $K' \subseteq K$. We thus end up with  $d_{B_1}(x_i)\in\{0,2\}$ for all $i \in K$, $d_{B_1}(x_{k+1})=1$, and $d_{B_1}(u)=\lambda g(u)$ for all $u\in V_{\alpha}$. In addition, we have $m_{B_1}(x_i,V({\mathcal H_i}))=1$ for all $i \in K'$.

The new graph $\mathcal G'$ and function $g':V(\mathcal G') \rightarrow {\mathbb Z}^+$ are now defined exactly as in the proof of Theorem~\ref{maininductiveseq}, and Properties (P1)--(P4c) are verified in the same way.

To see that Property (P4d) holds for $\mathcal G'$, observe that $|E(\mathcal G'(k+1))|=|E(\mathcal G(k+1))|=\frac{1}{2} mn$ by the induction hypothesis.

Finally, to verify Property (P4e), first observe that $d_{{\mathcal G}'(k+1)}(v)=d_{{\mathcal G}(k+1)}(v)=g'(v)$ for each $v\in V_{\alpha}$. Furthermore, since $d_{B_1}(x_{k+1})=1$, we have
$d_{{\cal G}'(k+1)}(\beta)=1=g'(\beta)$  and
$d_{{\cal G}'(k+1)}(\alpha)=g(\alpha)-1=g'(\alpha)$.
\end{proof}

We are now ready to prove our main Theorem~\ref{mainthmultialspach}.

\bigskip

\begin{proofof}~\ref{mainthmultialspach}.
Assume there exists a $(c_1,\dots,c_k)$-cycle decomposition of $\lambda m K_n$.

First, let $\lambda m (n-1)$ be even. By Theorem~\ref{maininductiveseq}, there exist a graph $\mathcal G=(V,E)$ of order $mn$ and a function $g:V \rightarrow {\mathbb Z}^+$ satisfying Properties (P1)--(P4). Thus $\mathcal G$ is $n$-partite, and since $\sum_{v\in V}g(v)=mn$ by Property (P2), and $g(v) \ge 1$ for all $v \in V$, we must have that $g(v) = 1$ for all $v \in V$. Therefore, again by Property (P2), each part of $\mathcal G$ has $m$ vertices. Furthermore, by Property (P3), $m_{\mathcal G}(u,v)= \lambda g(u)g(v)=\lambda$ for every pair of vertices in distinct parts of $\mathcal G$, implying that $\mathcal G$ is isomorphic to $\lambda K_{n\times m}$. By Properties (P4a) and (P4c), the graph $\mathcal G$ admits a $k$-edge-colouring such that each colour class ${\mathcal G}(i)$ has a unique non-trivial connected component ${\mathcal C}(i)$ with $c_im$ edges. Furthermore, property (P4b) tells us that each ${\mathcal C}(i)$ is 2-regular. Hence it is a cycle of length $c_im$. We thus have a $(c_1m,\dots,c_km)$-cycle decomposition of $\lambda K_{n\times m}$  as claimed.

When $\lambda m (n-1)$ is odd, Corollary~\ref{cor:maininductiveseq} similarly implies a decomposition of $\lambda K_{n\times m}$ into cycles of lengths $c_1 m,\ldots,c_k m$ plus a 1-factor. In particular, Properties (P4d)--(P4e) show that for ${\cal G}=\lambda K_{n\times m}$, since $g(v) = 1$ for all $v \in V$, the colour class ${\cal G}(k+1)$ is a 1-factor.
\end{proofof}

\section{More cycle decompositions of $\lambda K_{n \times m}$}

The proof of Theorem~\ref{maininductiveseq} was presented in the most general form, which we hope can be used in the future to derive other decomposition results. In the next theorem, however, we exploit the fact that at each step, each colour class is a detachment of an $m$-fold cycle (plus isolated vertices). This approach will allow us to obtain cycle decompositions of $\lambda K_{n \times m}$ with other cycle lengths. But first, we present the following simple observation, to be used in the proofs of Theorems~\ref{additional-1} and \ref{additional-2} below.

\begin{lemma}\label{lem:flower}
Let  $\cal G$ be a simple graph with a decomposition ${\cal C}=\{ C_1,\ldots,C_m \}$ into $c$-cycles. Assume there exists a vertex $\alpha$ of $\cal G$ such that each pair of distinct cycles in $\cal C$ intersect only in this vertex $\alpha$. Then:
\begin{enumerate}
\item There exists an $\alpha$-detachment $\cal H$ of $\cal G$ obtained by splitting vertex $\alpha$ into $m$ vertices so that $\cal H$ is isomorphic to a $cm$-cycle.
\item If $q_1,\ldots,q_r$ are positive integers such that $\sum_{i=1}^r q_i=m$, then there exists an $\alpha$-detachment $\cal H$ of $\cal G$ obtained by splitting vertex $\alpha$ into $m$ vertices so that $\cal H$ is a vertex-disjoint union of cycles of lengths $q_1c,\ldots,q_rc$.
\end{enumerate}
\end{lemma}

\begin{theo}\label{additional-1}
Assume there exists a $(c_1, \dots,c_k)$-cycle decomposition of $\lambda m K_n$. Let $C_i$ be the cycle of length $c_i$ in this decomposition, and assume that the cycles $C_1,C_2,\ldots,C_k$ have been ordered so that for some integer $N$ with $1 \le N  < k$, for all $i \le N$, the cycle $C_i$ possesses a vertex not in $ \bigcup_{j=1}^{i-1} V(C_j)$.

For all $i \in \{1,\ldots,N\}$, let $r_i$ and $q_{i,1},\ldots,q_{i,r_i}$ be positive  integers such that $q_{i,1}+\ldots +q_{i,r_i}=m$, and if $\lambda=1$ and $c_i=2$, then each $q_{i,j} \ge 2$. Then the complete equipartite multigraph $\lambda K_{n \times m}$ can be decomposed into subgraphs $F_1,\ldots,F_k$, plus a 1-factor if $\lambda m (n-1)$ is odd, such that
\begin{description}
\item[(R1)] for all $i \in \{1,\ldots,k\}$, the subgraph $F_i$ is 2-regular and has $c_im$ edges, and
\item[(R2)] for all $i \in \{1,\ldots,N\}$, the subgraph $F_i$ is a vertex-disjoint union of cycles of lengths $q_{i,1}c_i,\ldots,q_{i,r_i}c_i$.
\end{description}
\end{theo}

\begin{proof}
We shall first prove the theorem for the case that $\lambda m (n-1)$, the degree of $\lambda m K_n$, is even. As in the statement of the theorem, let $C_1,C_2,\ldots,C_k$ be an ordering of the cycles in a decomposition of $\lambda m K_n$ with the specified properties. We may assume that $N$ is maximum in the sense that for all $i>N$, the vertex set of $C_i$ is contained in $\bigcup_{j=1}^{i-1} V(C_j)$.

First, we obtain a decomposition of $\lambda m^2 K_n$ into $m$-fold cycles $C_1',\ldots,C_k'$ of lengths $c_1, \dots,c_k$ by replacing each edge in each cycle $C_i$ with $m$ parallel edges. Note that $V(C_i)=V(C_i')$ for all $i=1,\ldots k$.
Let $\cal G_0$ be the graph $\lambda m^2 K_n$ with a $k$-edge colouring $f_0$ arising from this cycle decomposition; that is, the colour class $\cal G_0(i)$ consists of the $m$-fold cycle $C_i'$ together with the remaining $n-c_i$ isolated vertices.

For each $s=1,2,\ldots,N$, we shall now  construct a graph $\cal G_s$ and its $k$-edge colouring $f_s$ with the following properties (to be verified below):
\begin{description}
\item[(R0$'$)] $\cal G_s$ is obtained from $\cal G_{s-1}$ by splitting each vertex in $V(C_s)-\bigcup_{j=1}^{s-1} V(C_j)$ into $m$ vertices;
\item[(R1$'$)] for all $i \in \{1,\ldots,k\}$, the colour class $\cal G_s(i)$ has $c_im$ edges;
\item[(R1$''$)] for all colours $i \in \{1,\ldots,k\}$ and all vertices $x \in  \bigcup_{j=1}^{s} V(C_j)$, if $x^{(j)}$ denotes a vertex of $\cal G_s$ obtained by splitting $x$, then  $d_{\cal G_s(i)}(x^{(j)}) \in \{0,2\}$; and
\item[(R2$'$)] for all $i \in \{1,\ldots,s\}$, the colour class $\cal G_s(i)$  is a vertex-disjoint union of cycles of lengths $q_{i,1}c_i,\ldots,q_{i,r_i}c_i$, plus isolated vertices.
\end{description}
Fixing $s\in \{1,\ldots, N\}$, we construct $\cal G_{s}$ from $\cal G_{s-1}$  as follows.
\begin{enumerate}
\item Choose some $v_s \in V(C_s)-\bigcup_{j=1}^{s-1} V(C_j)$.
\item For each $x \in V(C_s)-\bigcup_{j=1}^{s-1} V(C_j)$ such that $x \ne v_s$, do the following:
\begin{enumerate}
\item Split $x$ into $m$ vertices $x^{(1)},\ldots, x^{(m)}$.
\item For each vertex $y$ in $\cal G_{0}$, $y \ne x$, that is also a vertex of $\cal G_{s-1}$ (that is, $y$ has not been split yet), replace the set of $\lambda m^2$ parallel edges $xy$ (partitioned into $\lambda m$ colour classes of size $m$) with a decomposition of $\lambda mK_{1,m}$ into $\lambda m$ copies of $K_{1,m}$. That is, each of the $\lambda m$ colour classes of edges with one endpoint $y$ and the other in $\{x^{(1)},\ldots, x^{(m)} \}$ is isomorphic to $K_{1,m}$.
\item For each vertex $y$ in $\cal G_{0}$ that has already been split into vertices $y^{(1)}, \ldots, y^{(m)}$ of $\cal G_{s-1}$, the subgraph of $\cal G_{s-1}$ induced by the vertex set $\{ x,y^{(1)}, \ldots, y^{(m)} \}$ is (by the previous paragraph) isomorphic to $\lambda mK_{1,m}$ decomposed into $\lambda m$ colour classes $K_{1,m}$. After splitting the vertex $x$, replace this induced subgraph with a 1-factorization of $\lambda K_{m,m}$ (that is, each of the $\lambda m$ colour classes is a 1-factor in $\lambda K_{m,m}$).
\end{enumerate}
\item Observe that at this point, all vertices of $C_s$ except $v_s$ have been split, and the $s$-th colour class consists of $m$ cycles of length $c_s$ joined at a single vertex, namely, $v_s$.
\item Let $w_s$ be a vertex adjacent to $v_s$ in $C_s$. Split vertex $v_s$ into $m$ vertices $v_s^{(1)},\ldots, v_s^{(m)}$, and repeat  Steps (2b) and (2c) for $x=v_s$ and all vertices $y \ne w_s$.
\item Observe that vertex $w_s$ has already been split into vertices $w_s^{(1)},\ldots, w_s^{(m)}$. Replace the subgraph induced by the vertex set $\{v_s,w_s^{(1)},\ldots, w_s^{(m)} \}$, which is isomorphic to $\lambda mK_{1,m}$ decomposed into $\lambda m$ colour classes $K_{1,m}$, with a 1-factorization of $\lambda K_{m,m}$, first choosing the edges of the 1-factor corresponding to colour class $s$ so that the $s$-th colour class becomes a vertex-disjoint union of cycles of lengths $q_{s,1}c_s,\ldots,q_{s,r_s}c_s$ (this is possible by Lemma~\ref{lem:flower}).
\end{enumerate}
We shall now verify that graphs $\cal G_{1}, \ldots, \cal G_{N}$ satisfy Properties (R0$'$)--(R2$'$). From the construction, it is clear that Properties (R0$'$) and (R1$'$) hold; the latter holds since for each colour $i$, the number of edges of colour $i$ in $\cal G_{s-1}$ and $\cal G_{s}$ are equal. To see Property (R1$''$), observe that the degree of any vertex $x$ in $\cal G_0(i)$, for any colour $i \in \{ 1,\ldots, k\}$, is in $\{ 0, 2m \}$, and when vertex $x$ is split into vertices $x^{(1)},\ldots, x^{(m)}$, the degree of each vertex $x^{(j)}$ in the $i$-th colour class will be in $\{ 0,2 \}$. Since in $\cal G_s$, all vertices from $C_1,\ldots,C_s$ have already been split, Property (R1$''$) follows. Lastly, Step 5 guarantees that the non-trivial connected components of the colour class $\cal G_s(s)$  are cycles of lengths $q_{s,1}c_s,\ldots,q_{s,r_s}c_s$, while the non-trivial connected components of each colour class $\cal G_s(i)$ for $i<s$ are identical to those of $\cal G_{s-1}(i)$. Hence Property (R2$''$) holds as well.

Because $V(C_i) \subseteq \bigcup_{j=1}^{N} V(C_j)$ for all $i>N$, we have $\bigcup_{j=1}^{N} V(C_j)=V(\cal G_0)$. Thus,  all vertices in $\cal G_N$ have already been split (each into $m$ mutually non-adjacent vertices), and so $\cal G_N$ is isomorphic to $\lambda K_{n \times m}$. Moreover, by Properties (R1$'$) and (R1$''$), the $i$-th colour class (for $i \le k$) in $\cal G_N$ is a vertex-disjoint union of cycles with $c_im$ edges altogether. Finally, by Property (R2$'$), the colour class $\cal G_N(i)$, for all $i \in \{1,\ldots,N\}$,  is a vertex-disjoint union of cycles of lengths $q_{i,1}c_i,\ldots,q_{i,r_i}c_i$, plus isolated vertices. If we now define each $F_i$ (for $i=1,2,\ldots,k$) as the union of non-trivial connected components of the colour class $\cal G_N(i)$, then $\lambda K_{n \times m}$ has been decomposed into subgraphs $F_1,\ldots,F_k$ with Properties (R1)--(R2) as required.

The proof for the case when $\lambda m(n-1)$ is odd is very similar, so we shall only highlight the differences. We start with a presumed decomposition of $\lambda m K_n$ into cycles $C_1,C_2,\ldots,C_k$ (with the specified ordering) and a 1-factor. From this, we obtain a decomposition of $\lambda m^2 K_n$ into $m$-fold cycles and an $m$-fold 1-factor. We then let $\cal G_0$ be the graph $\lambda m^2 K_n$ with a $(k+1)$-edge colouring $f_0$ arising from this decomposition; that is, the colour class $\cal G_0(i)$, for $i \le k$, consists of the $m$-fold cycle of length $c_i$ together with the remaining $n-c_i$ isolated vertices, while the colour class $\cal G_0(k+1)$ is an $m$-fold 1-factor.

We then proceed to construct, for $s=1,2,\ldots,N$, a graph $\cal G_s$ and its $(k+1)$-edge colouring $f_s$ satisfying Properties (R0$'$)--(R2$'$), where Properties (R1$'$) and (R1$''$) are modified as follows:
\begin{description}
\item[(R1$'$)] for all $i \in \{1,\ldots,k\}$, the colour class $\cal G_s(i)$ has $c_im$ edges, while $\cal G_s(k+1)$ has $\frac{mn}{2}$ edges; and
\item[(R1$''$)] for all colours $i \in \{1,\ldots,k+1\}$ and all vertices $x \in  \bigcup_{j=1}^{s} V(C_j)$, if $x^{(j)}$ denotes a vertex of $\cal G_s$ obtained by splitting $x$, then  $d_{\cal G_s(i)}(x^{(j)}) \in \{0,2\}$ for $i \le k$, and $d_{\cal G_s(k+1)}(x^{(j)})=1$.
\end{description}
The construction (Steps 1-5) is performed exactly as in the first case, and the verification is very similar and hence left to the reader.
\end{proof}

In Theorem~\ref{additional-1}, we were able to split $N$ of the 2-regular subgraphs into cycles of desired lengths (divisible by the corresponding $c_i$). How large can $N$ be? Since for all $i=2,\ldots,N$, the cycle $C_i$ adds at least one vertex to $\bigcup_{j=1}^{i-1} V(C_j)$, we must have
$$c_1+(N-1) \le \big\vert \bigcup_{j=1}^{N} V(C_j) \big\vert \le n.$$ Hence $N \le n-c_1+1$. In the next example we describe a case in which Theorem~\ref{additional-2} (to follow below)  will give an improvement.

\begin{example}\rm{
Let $n$ be odd, each $c_i \ge 3$, and $\sum_{i=1}^k c_i={n \choose 2}$. Then, by Theorem~\ref{BryHorsPett},  there exists a decomposition of $K_n$ into cycles $C_1^*,C_2^*,\ldots,C_k^*$, where cycle $C_i^*$ is of length $c_i$, for all $i$. Taking each of these cycles with multiplicity $m$ we obtain a decomposition of $mK_n$ into $mk$ cycles of lengths $c_1,\ldots,c_1,c_2,\ldots,$ $c_2,\ldots,c_k,\ldots,c_k$ ($m$ repetitions of each $c_i$). If we now find an ordering $C_1,\ldots,C_{mk}$ of these cycles, and an index $N$ such that the assumptions of Theorem~\ref{additional-1} are satisfied, then $N \le n-c_1+1$ by the preceding paragraph. In the next result, we'll show that we can do better: at least $k$ of the 2-regular subgraphs in the decomposition of $K_{n \times m}$ that results by splitting each vertex in $mK_n$ into $m$ vertices will consist of cycles of specified lengths. When most cycle lengths among $c_1, \dots,c_k$ are small, their number $k$ may be quadratic in $n$, and so we'll have $n-c_1+1 < k$ for $n$ sufficiently large.
}
\end{example}

\begin{theo}\label{additional-2}
Let $n$ be odd, each $c_i \ge 3$, and $\sum_{i=1}^k c_i={n \choose 2}$. For all $i \in \{1,\ldots,k\}$, let $r_i$ and $q_{i,1},\ldots,q_{i,r_i}$ be positive  integers such that $q_{i,1}+\ldots +q_{i,r_i}=m$. Then the complete equipartite graph $K_{n \times m}$ can be decomposed into subgraphs $F_1,\ldots,F_{mk}$, plus a 1-factor if $n$ is even, such that
\begin{description}
\item[(S1)] for $i=1,\ldots,k$ and $j=0,\ldots,m-1$, the subgraph $F_{jk+i}$ is 2-regular and has $c_im$ edges, and
\item[(S2)] for $i \in \{1,\ldots,k\}$, the subgraph $F_{i}$ is a vertex-disjoint union of cycles of lengths $q_{i,1}c_i,\ldots,q_{i,r_i}c_i$.
\end{description}
\end{theo}

\begin{proof}
By Theorem~\ref{BryHorsPett}, there exists a $(c_1,\ldots,c_k)$-cycle decomposition of $K_n$. Taking $m$ copies of each of these cycles, we obtain a decomposition of $mK_n$ into $mk$ cycles of lengths $c_1,\ldots,c_1,c_2,\ldots,c_2,\ldots,c_k,\ldots,c_k$ ($m$ repetitions of each $c_i$). Label these cycles by $C_1,\ldots,C_{mk}$ and order them in such a way that the length of cycle $C_{jk+i}$, for $i=1,\ldots,k$ and $j=0,\ldots,m-1$, is $c_i$. Note that with this ordering, each pair of distinct vertices $x,y$ of $m K_n$ are adjacent in exactly one of the cycles $C_1, \ldots,C_k$.

Next, we obtain a decomposition of $\cal G=m^2 K_n$ into $m$-fold cycles $C_1',\ldots,C_{mk}'$ by replacing each edge in each cycle $C_i$ with $m$ parallel edges. Let $f$ be an $mk$-edge colouring of $\cal G$ arising from this cycle decomposition; that is, the colour class $\cal G(i)$ consists of the $m$-fold cycle $C_i'$ (with $m|V(C_i')|$ edges) together with the remaining $n-|V(C_i')|$ isolated vertices. Observe that for each pair of distinct vertices $x,y$, the subgraph of $\cal G$ induced by $\{ x,y\}$ is $m^2K_2$ decomposed into $m$ colour classes, each isomorphic to $mK_2$.

Obtain a graph $\cal G^*$ with an $mk$-edge colouring $f^*$ from $\cal G$ and $f$ as follows. First, split every vertex $x$ in $\cal G$ into $m$ vertices $x^{(1)},\ldots,x^{(m)}$. Then, for $s=1,\ldots,k$, ``lift'' the edges of the cycle $C_s$ as follows:
\begin{enumerate}
\item Choose an edge $uv$ of the cycle $C_s$.
\item For each pair $\{x,y\}$ of vertices adjacent in $C_s$, such that $\{x,y\} \ne \{ u,v\}$, perform the following operation on the subgraph of $\cal G$ induced by $\{x,y\}$: replace the decomposition of $m^2 K_2$ (subgraph of $\cal G$)
    into $m$ colour classes isomorphic to $mK_2$ with a 1-factorization of $K_{m,m}$ (subgraph of $\cal G^*$) so that each of the $m$ colour classes is now a 1-factor with $m$ edges.
\item Replace the subgraph of $\cal G$ induced by $\{ u,v\}$  with a 1-factorization of $K_{m,m}$ (subgraph of $\cal G^*$), first choosing the 1-factor corresponding to colour class $s$ so that the resulting colour class in $\cal G^*$ (induced by the set of vertices obtained by splitting all vertices of $C_s$) is a vertex-disjoint union of cycles of lengths $q_{s,1}c_s,\ldots,q_{s,r_s}c_s$ (this is possible by Lemma~\ref{lem:flower}).
\end{enumerate}
First, it is clear from the construction that $\cal G^*$ is isomorphic to $K_{n \times m}$. Next, observe that each colour class of $\cal G$ gives rise to a colour class in $\cal G^*$ of the same size; that is, each cycle $C_{jk+i}$ of length $c_i$ in $mK_n$ gives rise to an $m$-fold cycle $C_{jk+i}'$ in $\cal G$ with $c_im$ edges, which  gives rise to a colour class $\cal G^*(jk+i)$ with $c_im$ edges in $K_{n \times m}$. It is also easy to see that every vertex of $\cal G^*$ has degree 0 or 2 in each colour class $\cal G^*(jk+i)$. Hence the non-trivial connected components of $\cal G^*(jk+i)$  will form a 2-regular graph $F_{jk+i}$ with $c_im$ edges, thus satisfying Property (S1). Furthermore, Step 3 of the construction ensures that the 2-regular subgraphs $F_1,\ldots,F_k$ will consists of cycles of specified lengths, yielding Property (S2).
\end{proof}

Observe that attempting to extend the proof of Theorem~\ref{additional-2} to the case $n$ is even (that is, starting with a decomposition of $K_n$ into cycles plus a 1-factor) results in a decomposition of $K_{n \times m}$ into 2-regular subgraphs satisfying Properties (S1)--(S2) plus $\frac{n}{2}$ copies of $K_{m,m}$, rather than a 1-factor.

\section{Cycle decompositions of $\lambda K_{n \times m}$ for $2 \le n \le 4$}

In this section, we use Theorem~\ref{mainthmultialspach} to find particular cycle decompositions of complete equipartite multigraphs with a small number of parts, most of which were not known before. In the next three lemmas, we first construct all possible decompositions of $\mu K_n$, for $n=2,3,4$, into cycles of variable lengths. The first of these three lemmas is obvious, hence the proof is omitted.

In this section, a decomposition of a graph $\cal G$ into $m_1$ cycles of length $c_1$, $m_2$ cycles of length $c_2$, $\ldots$, and $m_\ell$ cycles of length $c_{\ell}$, plus a 1-factor if each vertex in $\cal G$ is of odd degree, will be abbreviated as $(c_1^{(m_1)},\ldots,c_{\ell}^{(m_{\ell})})$-CD.

\begin{lemma}\label{lem:n=2}
There exists a $(2, 2, \ldots, 2)$-CD of $\mu K_{2}$.
\end{lemma}

It is not difficult to verify that the necessary and sufficient conditions in Lemmas~\ref{lem:n=3} and \ref{lem:n=4} below are equivalent to Conditions (B1)--(B4) from Section 1 (for $n=3$ and $n=4$, respectively), however, they do not imply the conditions of Theorem~\ref{BryHorMaeSmi}.

\begin{lemma}\label{lem:n=3}
There exists a $(2^{(a)},3^{(b)})$-CD of $\mu K_{3}$ if and only if $2a+3b=3 \mu$.
\end{lemma}

\begin{proof}
Counting the edges in all cycles, we can see that if there exists a $(2^{(a)},3^{(b)})$-CD of $\mu K_{3}$, then $2a+3b=3 \mu$.

Conversely, if $2a+3b=3 \mu$, then we can find a required decomposition as follows. Take $b$ cycles of length 3; this is possible since $b=\frac{1}{3}(3\mu -2a)\le \mu$. We now have $\mu-b$ (necessarily an even number) of parallel edges left over between each pair of distinct vertices. Hence the remaining edges can be partitioned into $\frac{1}{2}(\mu {3 \choose 2}-3b)=a$  cycles of length 2.
\end{proof}

\begin{lemma}\label{lem:n=4}
There exists a $(2^{(a)},3^{(b)},4^{(c)})$-CD of $\mu K_{4}$ if and only if
\begin{itemize}
\item $2a+3b+4c=6 \mu-\delta$, where $\delta=0$ if $\mu$ is even, and $\delta=2$ if $\mu$ is odd; and
\item $(b,c) \ne (0,1)$ if $\mu$ is even, and $(b,c) \ne (0,0)$ if $\mu$ is odd.
\end{itemize}
\end{lemma}

\begin{proof}
It is easy to verify the necessity of the conditions of the lemma.

Conversely, assume that $\mu$, $a$, $b$, and $c$ satisfy the two conditions of the lemma. We shall construct a $(2^{(a)},3^{(b)},4^{(c)})$-CD of $\mu K_{4}$  as follows. In addition to graphs $C_3$, $C_4$, and $K_4$ with standard symbols, the basic building blocks of this decomposition will be obtained from the following graphs:
\begin{itemize}
\item $2(K_4-P_1)$, the graph $2K_4$ with a pair of parallel edges removed (this graph five edges, each of multiplicity 2);
\item $K_4+2P_1$, the graph $K_4$ with an added pair of parallel edges (this graph has one edge of multiplicity 3 and five edges of multiplicity 1);
\item  $K_4+2I_4$, the graph $K_4$ with an added 2-fold 1-factor (this graph has two independent edges of multiplicity 3 and four edges of multiplicity 1);
\item $3K_4-2I_4$, the multigraph $3K_4$ with the edges of a 2-fold 1-factor removed (this graph has four edges of multiplicity 3 and two independent edges of multiplicity 1); and
\item $3K_4-2P_2$, the multigraph $3K_4$ with the edges of a 2-fold path of length two removed (this graph has four edges of multiplicity 3 and two adjacent edges of multiplicity 1).
\end{itemize}
Observe that each of these (multi)graphs has either no vertices of odd degree or no vertices of even degree.

It is easy to establish existence of the following auxiliary decompositions, to be used below. For convenience, we make a note of which of these include a 1-factor.
\begin{description}
\item[(D1)] a $(3^{(4)})$-CD of $2K_4$;
\item[(D2)] a $(4^{(3)})$-CD of $2K_4$;
\item[(D3)] a $(4^{(2)})$-CD of $2C_4$;
\item[(D4)] a $(3^{(2)})$-CD of $2C_3$;
\item[(D5)] a $(4^{(1)})$-CD of $K_4$ (includes a 1-factor);
\item[(D6)] a $(3^{(2)},4^{(1)})$-CD of $2(K_4-P_1)$;
\item[(D7)] a $(3^{(2)})$-CD of $K_4+2P_1$ (includes a 1-factor);
\item[(D8)] a $(4^{(2)})$-CD of $K_4+2I_4$ (includes a 1-factor);
\item[(D9)] a $(4^{(3)})$-CD of $3K_4-2I_4$ (includes a 1-factor); and
\item[(D10)] a $(3^{(4)})$-CD of $3K_4-2P_2$ (includes a 1-factor).
\end{description}
Observe that $b$ must be even. Let $b=4b'+b''$ and $c=3c'+c''$ where $b'' \in \{ 0,2 \}$ and $c'' \in \{ 0,1,2 \}$. In most cases (that is, unless stated otherwise), we start with a $(3^{(4b')})$-CD and $(4^{(3c')})$-CD of edge-disjoint subgraphs $2b' K_4$ and $2c' K_4$, respectively, of $\mu K_4$, using Decompositions D1 and D2. This is possible since
$$\eta=\mu-(2b'+2c')=\frac{1}{6}(6\mu-3(b-b'')-4(c-c'')) = \frac{1}{6}(2a+\delta+3b''+4c'') \ge 0.$$
Observe that $\eta$ is even if and only if $\mu$ is even. We are left to construct a $(2^{(a)},3^{(b'')},4^{(c'')})$-CD of  the complete multigraph $\eta K_4$.  Depending on the values of $\eta$, $b''$, and $c''$, this task can be accomplished as follows.

{\em Case 1: $b''=2$ and $c''=2$.} If $\eta$ is even, then $\eta\ge 4$. First find a $(4^{(2)})$-CD (D3) and $(3^{(2)})$-CD (D4) of edge-disjoint subgraphs $2C_4$ and $2C_3$, respectively, of $\eta K_4$.

If $\eta$ is odd, then $\eta\ge 3$. In $\eta K_4$, first find a $(4^{(1)})$-CD (D5) and $(3^{(2)},4^{(1)})$-CD (D6) of edge-disjoint subgraphs $K_4$ and $2(K_4-P_1)$, respectively.

In both subcases, we are left with an even number of edges between each pair of distinct vertices, so the remaining edges can be partitioned into $\frac{1}{2}(\mu {4 \choose 2} - \delta- (3(4b'+2)+4(3c'+2)))=a$ cycles of length 2.

{\em Case 2: $b''=2$ and $c''=1$.} If $\eta$ is even, then $\eta \ge 2$. First find a $(3^{(2)},4^{(1)})$-CD of $2(K_4-P_1)$ (D6). If $\eta$ is odd, then $\eta \ge 3$. First find a $(4^{(1)})$-CD of $K_4$ (D5) and a $(3^{(2)})$-CD of $2C_3$ (D4).

{\em Case 3: $b''=2$ and $c''=0$.} If $\eta$ is even, we have $\eta \ge 2$. First find a $(3^{(2)})$-CD of $2C_3$ (D4). If $\eta$ is odd, we have $\eta \ge 3$. First find a $(3^{(2)})$-CD of $K_4+2P_1$ (D7).

{\em Case 4: $b''=0$ and $c''=2$.} If $\eta$ is even, then $\eta \ge 2$. First find a $(4^{(2)})$-CD of $2C_4$ (D3). If $\eta$ is odd, then $\eta \ge 3$. First find a $(4^{(2)})$-CD of $K_4+2I_4$ (D8).

{\em Case 5: $b''=0$ and $c''=1$.} If $\eta$ is even, then $\eta \ge 2$. Observe that $b' \ge 1$ or $c' \ge 1$ since $(b,c) \ne (0,1)$. We now need to modify the initial  decomposition as follows. If $c' \ge 1$, start with a $(3^{(4b')})$-CD of $2b' K_4$ (D1) and $(4^{(3(c'-1))})$-CD of $2(c'-1) K_4$ (D2). We are left with $(\eta+2) K_4$, and we proceed with a $(4^{(4)})$-CD of $4C_4$ (D3). If $c'=0$, then $b'\ge 1$; start with a $(3^{(4(b'-1))})$-CD of $2(b'-1) K_4$ (D1). We are left with $(\eta+2) K_4$, and we proceed with a $(3^{(2)},4^{(1)})$-CD of $2(K_4-P_1)$ (D6) and a $(3^{(2)})$-CD of $2C_3$ (D4).

If $\eta$ is odd, then $\eta \ge 1$. First find a $(4^{(1)})$-CD of $K_4$ (D5).

{\em Case 6: $b''=0$ and $c''=0$.} If $\eta$ is even, then clearly $(2^{(a)})$-CD of $\eta K_4$ exists.

If $\eta$ is odd, then $\eta \ge 1$, and $b' \ge 1$ or $c' \ge 1$ since $(b,c) \ne (0,0)$. Again, we need to modify the initial  decomposition. If $c' \ge 1$, start with a $(3^{(4b')})$-CD of $2b' K_4$ (D1) and $(4^{(3(c'-1))})$-CD of $2(c'-1) K_4$ (D2). We are left with $(\eta+2) K_4$, and we proceed with a $(4^{(3})$-CD of $3K_4-2I_4$  (D9). If $c'=0$, then $b'\ge 1$; start with a $(3^{(4(b'-1))})$-CD of $2(b'-1) K_4$ (D1). We are left with $(\eta+2) K_4$, and we proceed with a $(3^{(4)})$-CD of $3K_4-2P_2$ (D10).

Cases 2-6 can be verified similarly to Case 1.
\end{proof}

We are now ready for the main result of this section. Note that Part 1 of Corollary~\ref{cor:smalln} below has been previously proved in \cite{LasAue}; we include it for completeness.

\begin{cor}\label{cor:smalln}
\begin{enumerate}
\item {\rm \cite{LasAue}} There exists a $(2m,\ldots,2m)$-CD  of $\lambda K_{2 \times m}$ (that is, a Hamilton cycle decomposition).
\item If $2a+3b=3\lambda m$, then there exists a $((2m)^{(a)},(3m)^{(b)})$-CD of $\lambda K_{3 \times m}$.
\item If $2a+3b+4c=6\lambda m - \delta$, where $\delta=0$ if $\lambda m$ is even, and $\delta=2$ if $\lambda m$ is odd, $(b,c) \ne (0,1)$ if $\lambda m$ is even, and $(b,c) \ne (0,0)$ if $\lambda m$ is odd, then there exists a $((2m)^{(a)},(3m)^{(b)},(4m)^{(c)})$-CD of $\lambda K_{4 \times m}$.
\end{enumerate}
\end{cor}

\begin{proof}
By Lemmas~\ref{lem:n=2}, \ref{lem:n=3}, and \ref{lem:n=4}, respectively, there exist the following:
\begin{enumerate}
\item a $(2,\ldots,2)$-CD of $\lambda m K_{2}$;
\item a $(2^{(a)},3^{(b)})$-CD of $\lambda m K_{3}$; and
\item a $(2^{(a)},3^{(b)},4^{(c)})$-CD of $\lambda m K_{4}$.
\end{enumerate}
The result then follows by Theorem~\ref{mainthmultialspach}.
\end{proof}

\begin{center}
{\large \bf Acknowledgement}
\end{center}

Our sincere thanks to the anonymous referees for very quick reading and thoughtful comments, in particular to the referee who generously contributed the present proof of Lemma~\ref{lem:B5}. The second author also wishes to acknowledge financial support by the Natural Sciences and Engineering Research Council of Canada (NSERC).

\end{document}